\documentclass[11pt]{amsart}
\usepackage[utf8]{inputenc}
\usepackage{graphicx}
\usepackage{amssymb}
\usepackage{amscd}
\usepackage{hyperref}
\usepackage{amsmath,amsthm}
\usepackage{latexsym}
\usepackage{caption}
\usepackage{color}
\usepackage{cite}

\newtheorem{theorem}{Theorem}
\newtheorem{lemma}[theorem]{Lemma}
\newtheorem{corollary}[theorem]{Corollary}
\newtheorem{definition}{Definition}
\newtheorem{remark}{Remark}
\newtheorem{proposition}[theorem]{Proposition}
\newtheorem{exa}{Example}

\def\Z{\mathbb{Z}}

\def\Q{\mathbb{Q}}

\begin{document}
\title[Differences of squares of upper-triangular  matrices]{Differences of two  squares of upper-triangular $2\times 2$ integer matrices}
\author[Andrej Dujella, Zrinka Franu\v{s}i\'{c}]{Andrej Dujella, Zrinka Franu\v{s}i\'{c}}

\address{
Department of Mathematics\\
Faculty of Science\\
University of Zagreb\\
Bijeni{\v c}ka cesta 30, 10000 Zagreb, Croatia}

\email{duje@math.hr,\ fran@math.hr}




\begin{abstract}
We consider the problem of characterizing upper-triangular matrices
$M=\begin{pmatrix}p&r\\0&q\end{pmatrix}\in M_2(\mathbb Z)$
which can be represented in the form $A^2-B^2$ with upper-triangular integer matrices $A$ and $B$ and  give a complete criterion in terms of representations of $p$ and $q$ as differences of two squares and an additional divisibility condition on $r$. Also, we give a complete classification of representable matrices in terms of congruence conditions on $p$, $q$, and $r$.
\end{abstract}

\maketitle

\section{Introduction and motivation}

The classical problem of deciding which integers are representable as a difference of two squares is well known: an integer $n$ can be written in the form $n=x^2-y^2$ with $x,y\in\mathbb Z$ if and only if $n\not\equiv 2\pmod 4$.  Analogous questions can be considered in other rings. For example, a Gaussian integer $z=a+bi$ can be represented as a difference of two squares of Gaussian integers if and only if $b$ is even and $(a,b)\not\equiv(2,2)\pmod4$. In \cite{duje-fra}, the analogous problem in rings of integers of quadratic fields $\Q(\sqrt d)$ was studied. More precisely, a complete characterization of  elements that  can be represented as a difference of two squares  is obtained for those integers $d$ that satisfy certain conditions expressed in terms of the solvability  of certain Pellian equations. 

One of the main motivations for studying representations as differences of squares comes from their close connection with the \textit{problem of the existence of Diophantine quadruples}, i.e., $D(n)$-quadruples -- sets of four elements in a commutative ring such that the product of any two distinct elements, increased by $n$, is a perfect square.

In several papers dealing with $D(n)$-quadruples in rings of algebraic integers, representations of elements as differences of two squares naturally arise as a preparatory step in the analysis. This approach appears, for example, in  \cite{zr-mis,zr-soldo,zr-borka,ljerka}, where the structure of elements representable as differences of squares plays an important role in determining the existence of $D(n)$-quadruples in various number fields.

In some commutative rings it turns out that the existence of a $D(n)$-quadruple is equivalent to the representability of $n$ as a difference of two squares, while in other settings this equivalence fails. Counterexamples to such a conjectural connection were recently obtained in \cite{cgh,cgh2,gup}. 

A further connection between differences of squares and $D(n)$-quadruples is given in \cite{az}, where it is shown that there is no polynomial $D(n)$-quadruple in $\Z[X]$ (a ring of polynomials with integer coefficients) for certain polynomials $n\in \Z[X]$ that cannot be represented as a difference of two squares of polynomials in $\Z[X]$.

These results motivate further study of representations as differences of squares in various algebraic structures. An interesting setting in which to study the existence of $D(n)$-tuples is the \textit{ring of upper triangular $2\times 2$ integer matrices}, i.e., matrices of the form
$$ T=\begin{pmatrix}a&b\\0&c\end{pmatrix},~ a,b,c\in\Z.$$
We denote this set by $UT_2(\Z)$. Note that $UT_2(\Z)$ forms a ring with respect to the usual matrix addition and multiplication. Although this ring is not commutative, squares of upper triangular matrices are particularly easy to compute. Indeed,
$$ T^2=\begin{pmatrix}a^2&b(a+c)\\0&c^2\end{pmatrix}.$$
So, this simple structure makes  $UT_2(\Z)$ a convenient framework for studying representations as differences of squares and their connection with $D(n)$-tuples.  Since representations as differences of two squares often play an important role in the study of $D(n)$-tuples, and the present paper provides a natural starting point for such investigations.  In forthcoming work, we intend to study the relationship between representability of an upper-triangular matrix $N$ as a difference of two squares and the existence of $D(N)$-quadruples of upper-triangular matrices with respect to the Jordan product
$$ A\circ B=\frac12(AB+BA).$$

For $A=\begin{pmatrix}a&b\\0&d\end{pmatrix}$ and 
$B=\begin{pmatrix}x&y\\0&u\end{pmatrix}$ in $UT_2(\Z)$, a direct computation gives
\begin{equation}\label{e.1}
    A^2-B^2=
\begin{pmatrix}
 a^2-x^2 & b(a+d)-y(x+u)\\
 0 & d^2-u^2
\end{pmatrix}.
\end{equation}
In view of the classical characterization over $\Z$, i.e. of integers representable as differences of two squares, relation \eqref{e.1} immediately yields a necessary condition for a matrix in $UT_2(\Z)$ to be expressible as a difference of two squares. Namely, if $T=\begin{pmatrix}p&r\\0&q\end{pmatrix}$ can be written in the form $A^2-B^2$ for some $A,B\in UT_2(\Z)$, then both diagonal entries $p$ and $q$ must be representable as differences of two integer squares, i.e.
$$ p\not\equiv2\pmod 4,~ q\not\equiv2\pmod 4.$$
Therefore, any upper-triangular integer matrix having at least one diagonal entry congruent to 2 modulo 4 cannot be represented as a difference of two squares in $UT_2(\Z)$. 

The upper-right entry in \eqref{e.1} introduces an additional arithmetic condition linking the representations of the diagonal entries. Our first main result (Theorem~\ref{thm:main-criterion}) shows that this condition reduces to the solvability of a linear Diophantine equation and can be expressed as a gcd divisibility condition. We conclude the paper with a complete classification of representable matrices in terms of congruence conditions on $p$, $q$, and $r$ (Theorem~\ref{thm:final-classification}).  The proof of the final classification combines the analysis of congruence obstructions modulo 4 and modulo 16 with explicit constructions of representations in the admissible cases (Propositions \ref{p1}, \ref{prop:odd-diagonal}, \ref{prop:mixed-cases}, \ref{prop:4k-4n-complete}).

 Analogous results also hold for the ring $LT_2(\Z)$ of lower-triangular $2\times 2$ integer matrices.

\section{The main criterion}

\begin{theorem}\label{thm:main-criterion}
The upper-triangular matrix
$T=\begin{pmatrix}p&r\\0&q\end{pmatrix}$ in $UT_2(\Z)$
can be represented as a difference of two squares in $UT_2(\Z)$  if and only if there exist integers $a,x,d,u$ such that
$$p=a^2-x^2,\qquad q=d^2-u^2,$$
and either
$$(a+d,x+u)\neq(0,0)\quad\text{and}\quad \gcd(a+d,x+u)\mid r,$$
or
$$a+d=0,\quad x+u=0,\quad r=0.$$
\end{theorem}
\begin{proof} $\boxed{\Rightarrow}$ : 
Assume  that $T=A^2-B^2$ for  upper-triangular matrices
$A=\begin{pmatrix}a&b\\0&d\end{pmatrix}$, $B=\begin{pmatrix}x&y\\0&u\end{pmatrix}$.
By \eqref{e.1}, we have 
$$p=a^2-x^2, \qquad q=d^2-u^2, \qquad r=b(a+d)-y(x+u).$$
If $(a+d,x+u)\neq(0,0)$, then $r$ is an integer linear combination of $a+d$ and $x+u$, and hence
$$\gcd(a+d,\,x+u)\mid r.$$
If $a+d=0$ and $x+u=0$, then the above relation reduces to $r=0$.

$\boxed{\Leftarrow}$ : Suppose there exist integers $a,x,d,u$ such that
$$p=a^2-x^2, \qquad q=d^2-u^2,$$
and either $(a+d,x+u)\neq(0,0)$ with $\gcd(a+d,x+u)\mid r$, or $a+d=x+u=0$ and $r=0$. If $(a+d,x+u)\neq(0,0)$, then the linear Diophantine equation
$$b(a+d)-y(x+u)=r$$
has an integer solution $(b,y)$. If $a+d=x+u=0$ and $r=0$, then any integers $b,y$ satisfy the equation. Defining $A=\begin{pmatrix}a&b\\0&d\end{pmatrix}$ and 
$B=\begin{pmatrix}x&y\\0&u\end{pmatrix}$, relation \eqref{e.1} yields $A^2-B^2=T$.
\end{proof}

\begin{corollary}\label{cor:diag-obstruction}
If $p\equiv 2\pmod 4$ or $q\equiv 2\pmod 4$, then $T=\begin{pmatrix}p&r\\0&q\end{pmatrix}$ in $UT_2(\Z)$ cannot be represented as a difference of two squares in $UT_2(\Z)$.
\end{corollary}

\begin{remark}
Theorem~\ref{thm:main-criterion} shows that the problem splits into two parts:
\begin{enumerate}
\item represent the diagonal entries $p$ and $q$ as differences of two squares;
\item choose such representations so that $\gcd(a+d,x+u)$ divides the upper-right entry $r$.
\end{enumerate}
Thus, the essential difficulty lies in the interaction between the chosen representations of $p$ and $q$. The greatest flexibility for the entry $r$ is obtained when this gcd is as small as possible.
\end{remark}

\begin{definition}\label{def:g}
For integers $p$ and $q$ representable as differences of two squares, define
$$g(p,q):=\min \bigl\{\gcd(a+d,x+u)>0\bigr\},$$
where the minimum is taken over all quadruples $(a,x,d,u)\in\mathbb Z^4$ satisfying
$$p=a^2-x^2, \qquad q=d^2-u^2.$$
\end{definition}

Here the condition $>0$ excludes the trivial case $(a+d,x+u)=(0,0)$. If for some representation one has $a+d=0$ and $x+u=0$, then replacing
$d$ and $u$ by $-d$ and $-u$ yields another representation of $q$
for which $(a+d,x+u)\neq(0,0)$. Hence the set in Definition \ref{def:g}
is nonempty.


\section{Complete classification by congruence type}

In several rings (for example $\mathbb Z$, $\mathbb Z[i]$, and rings of integers of certain number fields), it has proved useful to classify elements representable as differences of two squares according to their congruence classes. Such classifications often play an important role in constructions of $D(n)$-quadruples, since in some rings the existence of $D(n)$-quadruples is closely related to representability of $n$ as a difference of two squares. Motivated by this, we investigate which congruence classes modulo $4$ can occur for matrices representable as differences of two squares in $UT_2(\mathbb Z)$. For this purpose, we study the corresponding problem in
$UT_2(\Z_4)$, the ring of upper triangular $2\times 2$ matrices with entries in $\mathbb Z_4$.

 \begin{proposition}\label{p1}
The matrices in $UT_2(\mathbb Z_4)$ that cannot be represented as differences of two squares in $UT_2(\mathbb Z_4)$ are precisely
$$S=\left\{
\begin{pmatrix}
a & b\\
0 & c
\end{pmatrix}\in UT_2(\mathbb Z_4) : a=2 \text{ or } c=2 \right\}\bigcup 
\left\{ \begin{pmatrix}
1 & 1,3\\
0 & 1
\end{pmatrix}, \begin{pmatrix}
3 & 1,3\\
0 & 3
\end{pmatrix}\right\}. $$
\end{proposition}
\begin{proof}
We determine the set $S_1=\{A^2-B^2:\ A,B\in UT_2(\mathbb Z_4)\}$
using formula \eqref{e.1}. A direct verification shows that precisely the matrices listed in the statement do not belong to $S_1$. The complete lists of representable (32 cases) and non-representable (32 cases) matrices modulo $4$ are given in the Appendix.
\end{proof}

\begin{corollary}\label{cor:not_sq_mod_4}
    If $T\in UT_2(\mathbb Z)$ and $T \bmod 4 \in S$, then $T$ cannot be represented as a difference of two squares in $UT_2(\mathbb Z)$. 
\end{corollary}

However, the converse of Corollary \ref{cor:not_sq_mod_4}  does not hold: the condition $T \bmod 4 \notin S$ does not guarantee that $T$ is representable as a difference of two squares in $UT_2(\mathbb Z)$. So modulo 4 classification gives necessary but not sufficient conditions.

\begin{exa}\label{exa:N}
   The matrix 
$$N=\begin{pmatrix}
4 & 2\\
0 & 4\end{pmatrix}$$
is not representable as a difference of two squares in $UT_2(\mathbb Z)$, although
$$N \bmod 4 =\begin{pmatrix}
0 & 2\\
0 & 0\end{pmatrix}$$
is representable as a difference of two squares in $UT_2(\mathbb Z_4)$.
Indeed, suppose
$$N=
\begin{pmatrix}
a_1 & b_1\\
0 & c_1
\end{pmatrix}^2-
\begin{pmatrix}
a_2 & b_2\\
0 & c_2
\end{pmatrix}^2=\begin{pmatrix}
a_1^2-a_2^2 &
b_1(a_1+c_1)- b_2(a_2+c_2)\\
0 &
c_1^2-c_2^2
\end{pmatrix}.$$
From $a_1^2-a_2^2=4$ and $c_1^2-c_2^2=4$, we obtain
$$(a_1,a_2)\in\{(\pm2,0)\},\qquad
(c_1,c_2)\in\{(\pm2,0)\}.$$
Hence $a_1+c_1\in\{-4,0,4\}$ and $a_2+c_2=0$, so
$$b_1(a_1+c_1)-b_2(a_2+c_2)=b_1(a_1+c_1)$$
is divisible by $4$, which contradicts the condition that the upper-right entry equals $2$.
\end{exa}

\begin{proposition}\label{prop:odd-diagonal}
Let
$$M=\begin{pmatrix}2k+1&r\\0&2n+1\end{pmatrix}\in UT_2(\Z).$$
Then $M$ is representable as a difference of two squares in $UT_2(\Z)$ if and only if
$$\gcd(k+n,2)\mid r.$$
(Or equivalently, 
 if $k$ and $n$ have opposite parity, then every $r\in\mathbb Z$ is admissible, and  if $k$ and $n$ have the same parity, then $r$ must be even.)
\end{proposition}
\begin{proof} $\boxed{\Leftarrow}$:
We use the identity
\[
\begin{pmatrix}
k + 1& b\\ 0& n + 1
\end{pmatrix}^2-
\begin{pmatrix}
k & y\\ 0& n
\end{pmatrix}^2
=
\begin{pmatrix}
2k+1 & b(k+n+2)-y(k+n)\\
0 & 2n+1
\end{pmatrix}.
\]
Therefore, it suffices to determine when the linear Diophantine equation
\begin{equation}\label{e.a}
b(k+n+2)-y(k+n)=r
\end{equation}
is solvable in integers $b$ and $y$. This happens if and only if
\[
\gcd(k+n+2,k+n)\mid r.
\]
Since
\[
\gcd(k+n+2,k+n)=\gcd(k+n,2),
\]
equation \eqref{e.a} is solvable if and only if
\[
\gcd(k+n,2)\mid r.
\]
Hence $M$ is representable as a difference of two squares in $UT_2(\mathbb Z)$.

$\boxed{\Rightarrow}$:
Suppose that
\[
M=\begin{pmatrix}2k+1&r\\0&2n+1\end{pmatrix}
\]
is representable as a difference of two squares in $UT_2(\mathbb Z)$. Reducing modulo $4$, we obtain a matrix representable as a difference of two squares in $UT_2(\mathbb Z_4)$.

If $k$ and $n$ have the same parity, then the diagonal entries of $M \bmod 4$ are equal and odd, i.e. both are $1$ or both are $3$. By Proposition~\ref{p1}, such a matrix is representable modulo $4$ only if its upper-right entry is even. Hence $r$ must be even and since $\gcd(k+n,2)=2$, the divisibility condition holds.

If $k$ and $n$ have opposite parity, then $\gcd(k+n,2)=1$, so the desired divisibility condition is automatically satisfied.

Therefore, in all cases, $\gcd(k+n,2)\mid r.$
\end{proof}

\begin{corollary}
    For every $k,n,r\in\mathbb Z$,  the matrices 
    $$ \begin{pmatrix}4k+1&r\\0&4n+3\end{pmatrix},
    \begin{pmatrix}4k+3&r\\0&4n+1\end{pmatrix},
    \begin{pmatrix}4k+1&2r\\0&4n+1\end{pmatrix},
    \begin{pmatrix}4k+3&2r\\0&4n+3\end{pmatrix}$$
    are representable as differences of two squares in $UT_2(\Z)$.
\end{corollary}

\begin{proposition}\label{prop:mixed-cases}
For every $k,n,r\in\mathbb Z$, both matrices
\[
\begin{pmatrix}2k+1&r\\0&4n\end{pmatrix}
\qquad\text{and}\qquad
\begin{pmatrix}4k&r\\0&2n+1\end{pmatrix}
\]
are representable as differences of two squares in $UT_2(\Z)$.
\end{proposition}

\begin{proof}
For the first matrix, we take
$$a=k+1,\qquad x=k, \qquad d=-(n+1), \qquad u=-(n-1).$$
Then
$$a^2-x^2=2k+1,
\qquad d^2-u^2=4n,$$ 
and
$$a+d=k-n,\qquad x+u=k-n+1. $$
Hence
$$\gcd(a+d,x+u)=\gcd(k-n,k-n+1)=1,$$
so Theorem~\ref{thm:main-criterion} applies for every $r$.

The second case is symmetric. We take
$$a=-(k+1), \qquad x=-(k-1), \qquad d=n+1, \qquad u=n.$$
Then
$$a^2-x^2=4k, \qquad d^2-u^2=2n+1,$$
and
$$a+d=n-k,\qquad x+u=n-k+1,$$
so again $\gcd(a+d,x+u)=1$.

Thus both matrices are representable for every $r$.
\end{proof}

Proposition~\ref{prop:mixed-cases} shows that in the mixed parity cases, the upper-right entry imposes no additional restriction: every such matrix is representable as a difference of two squares in $UT_2(\mathbb Z)$.

\begin{lemma}\label{lemma:mod16-classification}
Let
\[
M=\begin{pmatrix}4i&r\\0&4j\end{pmatrix}\in UT_2(\mathbb Z_{16}),
\qquad i,j\in\{0,1,2,3\}.
\]
Then $M$ is representable as a difference of two squares in $UT_2(\mathbb Z_{16})$ if and only if one of the following holds:
\begin{enumerate}
\item $(i,j)\in\{(1,1),(3,3)\}$ and $r\equiv0\pmod4$;
\item $(i,j)\in\{(1,3),(2,2),(3,1)\}$ and $r\equiv0\pmod2$;
\item $(i,j)\in\{(0,0),(0,1),(0,2),(0,3),(1,0),(1,2),(2,0),(2,1),(2,3),(3,0),(3,2)\}$, with arbitrary $r\in\mathbb Z_{16}$.
\end{enumerate}
\end{lemma}
\begin{proof} The statement follows by direct computation. All representable matrices modulo $16$ are listed in the Appendix.
\end{proof}
Exactly 48 out of 256 matrices in $UT_2(\mathbb Z_{16})$ with diagonal entries divisible by 4 fail to be representable as differences of two squares, i.e. the proportion of non-representable matrices is 3/16 (see the Appendix).

\begin{corollary}\label{cor:nonrep-mod16}
Let
\[
M=\begin{pmatrix}4i&r\\0&4j\end{pmatrix}\in UT_2(\mathbb Z_{16}),
\qquad i,j\in\{0,1,2,3\}.
\]
Then $M$ cannot be represented as a difference of two squares in $UT_2(\mathbb Z_{16})$ if and only if one of the following holds:
\begin{enumerate}
\item$(i,j)\in\{(1,1),(3,3)\}$ and $r\not\equiv 0\pmod 4$;
\item$(i,j)\in\{(1,3),(2,2),(3,1)\}$ and $r\not\equiv 0\pmod 2$.
\end{enumerate}
\end{corollary}

\begin{proposition}\label{prop:4k-4n-complete}
Let
$$M=\begin{pmatrix}4k&r\\0&4n\end{pmatrix}\in UT_2(\Z).$$
Then $M$ is representable as a difference of two squares in $UT_2(\Z)$ if and only if one of the following conditions holds:
\begin{enumerate}
\item $k+n\equiv1\pmod 2$, or $(k,n)\bmod 4\in\{(0,0),(0,2),(2,0)\}$, with arbitrary $r\in\mathbb Z$;
\item $k+n\equiv0\pmod 4$, $(k,n)\bmod 4\neq(0,0)$, and $r\equiv0\pmod 2$;
\item $k+n\equiv2\pmod 4$,$(k,n)\bmod 4\notin\{(0,2),(2,0)\}$, and $r\equiv0\pmod 4$.
\end{enumerate}
\end{proposition}
\begin{proof}  $\boxed{\Leftarrow}$:
 First, the choice
$$a=k+1,\qquad x=k-1,\qquad d=n+1,\qquad u=n-1,$$
shows representability whenever $r$ satisfies the divisibility condition determined by $\gcd(k+n+2,4)$.
Indeed,
$$a^2-x^2=4k,\qquad d^2-u^2=4n,$$
and
$$a+d=k+n+2,\qquad x+u=k+n-2.$$
Hence
$$\gcd(a+d,x+u)=\gcd(k+n+2,k+n-2)=\gcd(k+n+2,4).$$
Therefore:
\begin{enumerate}
\item[(a)] if $k+n\equiv 1 \pmod 2$, then $\gcd(a+d,x+u)=1$;
\item[(b)] if $k+n\equiv 0 \pmod 4$, then $\gcd(a+d,x+u)=2$;
\item[(c)] if $k+n\equiv 2 \pmod 4$, then $\gcd(a+d,x+u)=4$.
\end{enumerate}
In each case, the assumed divisibility condition on $r$ implies that
$\gcd(a+d,x+u)\mid r.$
Hence Theorem~\ref{thm:main-criterion} shows that $M$ is representable as a difference of two squares in $UT_2(\Z)$ in the following cases:
\begin{itemize}
   \item $k+n\equiv1\pmod 2$,
\item $k+n\equiv0\pmod 4$ and $r\equiv0\pmod 2$;
\item $k+n\equiv2\pmod 4$ and $r\equiv0\pmod 4$.
\end{itemize}

For the remaining cases, we also give concrete representations.  For the case $k\equiv n\equiv 0 \pmod 4$, we put $k=4K$ and $n=4N$. For
$$A=\begin{pmatrix}2(K+1)&b\\0&-(4N+1)\end{pmatrix},\qquad
B=\begin{pmatrix}2(K-1)&y\\0&-(4N-1)\end{pmatrix},$$
we obtain
$$A^2-B^2=\begin{pmatrix}
16K & b(2K-4N+1)-y(2K-4N-1)\\0 & 16N
\end{pmatrix}.$$
Since
$$\gcd(2K-4N+1,\;2K-4N-1)=1,$$
the upper-right entry can be chosen arbitrarily. Hence every matrix
$$\begin{pmatrix}16K&r\\0&16N\end{pmatrix}$$
is representable as a difference of two squares in $UT_2(\Z)$.

Assume $k\equiv 0 \pmod 4$ and $n\equiv 2 \pmod 4$, say $k=4K$ and $n=4N+2$. Consider
$$A=\begin{pmatrix}2(K+1)&b\\0&-(4N+3)\end{pmatrix},\qquad
B=\begin{pmatrix}2(K-1)&y\\0&4N+1\end{pmatrix}.$$
Then
$$A^2-B^2=\begin{pmatrix}
16K &
b(2K-4N-1)+y(2K-4N-3)\\0 &8(2N+1)\end{pmatrix}.$$
Since
$$\gcd(2K-4N-1,\;2K-4N-3)=1,$$
the upper-right entry can be chosen arbitrarily. Hence every matrix of the form
$$\begin{pmatrix}
16K&r\\0&8(2N+1)
\end{pmatrix}$$
is representable as a difference of two squares in $UT_2(\mathbb Z)$. The case $(k,n)\equiv(2,0)\pmod 4$ is analogous.

$\boxed{\Rightarrow}$:
The claim follows from Corollary~\ref{cor:nonrep-mod16}, by contraposition of the fact that representability over $\mathbb Z$ implies representability modulo $16$.
\end{proof}


\begin{remark}
Example \ref{exa:N}  can now be explained in terms of Proposition~\ref{prop:4k-4n-complete}. Indeed,  case (3) of Proposition~\ref{prop:4k-4n-complete} requires $r\equiv0\pmod4$.
\end{remark}

\section{Final classification and comments}

We may now summarize the representation problem completely.

\begin{theorem}\label{thm:final-classification}
Let
$$M=\begin{pmatrix}p&r\\0&q\end{pmatrix}\in UT_2(\Z).$$

Then $M$ is representable as a difference of two squares in $UT_2(\Z)$ if and only if one of the following holds:
\begin{enumerate}
\item $p$ and $q$ are odd, and
\begin{itemize}
\item  $p\equiv q\pmod4$ and $r\equiv0\pmod2$;
\item $p\not\equiv q\pmod4$ and $r$ is arbitrary;
\end{itemize}
\item one of $p,q$ is odd and the other is divisible by $4$, and $r$ is arbitrary;
\item $p\equiv q\equiv0\pmod4$, and
\begin{itemize}
\item $(p,q)\mod16\in\{(4,4),(12,12)\}$ and $r\equiv0\pmod4$;
\item $(p,q)\mod16\in\{(4,12),(8,8),(12,4)\}$ and $r\equiv0\pmod2$;
\item otherwise, $r$ is arbitrary.
\end{itemize}
\end{enumerate}
\end{theorem}

\begin{proof}
The statement follows by combining the classification in the odd-diagonal case (Proposition~\ref{prop:odd-diagonal}), the mixed cases (Proposition~\ref{prop:mixed-cases}), and the case of diagonal entries divisible by $4$ (Proposition~\ref{prop:4k-4n-complete}), together  with Corollary \ref{cor:not_sq_mod_4}.
\end{proof}

The previous result shows that approximately one half of upper-triangular integer matrices are representable as differences of two squares.

\begin{corollary}
    The set of representable matrices has asymptotic density $125/256$ in $UT_2(\Z)$.
\end{corollary}
\begin{proof}
    According to Theorem \ref{thm:final-classification}, we can precisely list the following: 
    \begin{center}
    \begin{tabular}{c|c|c}
        & number of pairs    &  \\ 
         conditions on $(p,q)$  & $(p,q)$ mod 16  &  conditions on $r$   \\\hline \hline
 $p\equiv q\equiv 1\pmod2$,    $p\equiv q\pmod4$     & 32  & $r\equiv 0\pmod 2$ \\ \hline 
      $p\equiv q\equiv 1\pmod2$,    $p\not\equiv q\pmod4$     & 32  & no cond. \\ \hline
        $p\equiv 1\pmod2$,    $q \equiv 0\pmod4$     & 32  & no cond. \\ \hline
         $q\equiv 1\pmod2$,    $p \equiv 0\pmod4$     & 32  & no cond. \\ \hline
          $(p,q)\equiv (4,4),(12,12)\pmod{16}$     & 2  & $r\equiv 0\pmod 4$ \\ \hline
           $(p,q)\equiv (4,12),(8,8),(12,4)\pmod{16}$     & 3  & $r\equiv 0\pmod 2$ \\ \hline
     $p\equiv q\equiv 0\pmod4$, & & \\    
     different from the previous two       & $16-2-3=11$  & no cond. \\  
    \end{tabular}     
    \end{center}
Therefore, the density is $(\frac{1}{2}\cdot 32+3\cdot 32+\frac{1}{4}\cdot 2+\frac{1}{2}\cdot 3+11)/256=125/256$.
\end{proof}

\begin{theorem}\label{thm:g-pq}
The matrix $T=\begin{pmatrix}p&r\\0&q\end{pmatrix}\in UT_2(\mathbb Z)$ 
is representable as a difference of two squares in $UT_2(\mathbb Z)$ if and only if
\begin{enumerate}
\item $p$ and $q$ are representable as differences of two squares in $\mathbb Z$, and
\item $g(p,q)\mid r$, where $g(p,q)$ is given in Definition \ref{def:g}.
\end{enumerate}
\end{theorem}
\begin{proof}
Sufficiency follows immediately from Theorem~\ref{thm:main-criterion}.

For necessity, Theorem \ref{thm:final-classification} and the explicit
constructions in the proofs of Propositions \ref{prop:odd-diagonal}, \ref{prop:mixed-cases} and \ref{prop:4k-4n-complete} show that $T$ is
representable as a difference of two squares in $UT_2(\Z)$ if and only if $p$
and $q$ are representable as differences of two squares in $\Z$ and $m(p, q)\mid r$,
where $m(p, q) = \gcd(a+d, x+u)$ for suitable representations $p = a^2-x^2$, $q = d^2-u^2$. Hence,
$$g(p,q)\le m(p,q).$$
Observe that $m(p,q)\in\{1,2,4\}$. If $m(p,q)=1$, then clearly $g(p,q)=m(p,q)$. If $m(p,q)\in\{2,4\}$, suppose that $g(p,q)<m(p,q)$.Then, by
Theorem  \ref{thm:final-classification}, the matrix
$$\begin{pmatrix}p&g(p,q)\\0&q\end{pmatrix}$$
would be representable as a difference of two squares in $UT_2(\Z)$, contradicting Theorem \ref{thm:final-classification}. Therefore $g(p,q) = m(p,q)$, and the claim follows.
\end{proof}

The results obtained in this paper suggest further questions concerning the existence of Diophantine $D(n)$-tuples in noncommutative rings such as $UT_2(\Z)$. Another natural direction would be to investigate representability as a difference of two squares in the full matrix ring $M_2(\Z)$. Note that in \cite{kan} the analogous problem for sums of two squares of matrices in $M_2(\Z)$ is studied. It is interesting that a complete characterization for matrices in $M_2(\Z)$ that are representable as sums of two squares in $M_2(\Z)$   is given purely in terms of congruence conditions modulo 4.

\section*{Appendix: Tables of representations modulo 4 and 16}

\noindent\textbf{Notation.} We write $(a,b,c)$ for the matrix 
$\begin{pmatrix} a & b \\ 0 & c \end{pmatrix}$.

\noindent\underline{Matrices representable as a difference of two squares in $UT_2(\Z_4)$:}

{\tiny
\begin{tabular}{ccccccccccc}
(0,0,0) & (0,0,1) & (0,0,3) & (0,1,0) & (0,1,1) & (0,1,3) & (0,2,0) & (0,2,1) & (0,2,3) & (0,3,0) & (0,3,1) \\
(0,3,3) & (1,0,0) & (1,0,1) & (1,0,3) & (1,1,0) & (1,1,3) & (1,2,0) & (1,2,1) & (1,2,3) & (1,3,0) & (1,3,3) \\
(3,0,0) & (3,0,1) & (3,0,3) & (3,1,0) & (3,1,1) & (3,2,0) & (3,2,1) & (3,2,3) & (3,3,0) & (3,3,1) & \\
\end{tabular}}

\noindent\underline{Matrices not representable as a difference of two squares in $UT_2(\Z_4)$:}

{\tiny
\begin{tabular}{ccccccccccc}
(0,0,2)&(0,1,2)&(0,2,2)&(0,3,2)&(1,0,2)&(1,1,1)&(1,1,2)&(1,2,2)&(1,3,1)&(1,3,2)&(2,0,0)\\
(2,0,1)&(2,0,2)&(2,0,3)&(2,1,0)&(2,1,1)&(2,1,2)&(2,1,3)&(2,2,0)&(2,2,1)&(2,2,2)&(2,2,3)\\
(2,3,0)&(2,3,1)&(2,3,2)&(2,3,3)&(3,0,2)&(3,1,2)&(3,1,3)&(3,2,2)&(3,3,2)&(3,3,3)&
\end{tabular}
}

\noindent\underline{Matrices representable as a difference of two squares in $UT_2(\Z_{16})$, and $4\mid a,c$:}

{
\begin{tiny}
\begin{tabular}{ccccccccc}
(0,0,0)&(0,0,4)&(0,0,8)&(0,0,12)&(0,1,0)&(0,1,4)&(0,1,8)&(0,1,12)&(0,2,0)\\
(0,2,4)&(0,2,8)&(0,2,12)&(0,3,0)&(0,3,4)&(0,3,8)&(0,3,12)&(0,4,0)&(0,4,4)\\
(0,4,8)&(0,4,12)&(0,5,0)&(0,5,4)&(0,5,8)&(0,5,12)&(0,6,0)&(0,6,4)&(0,6,8)\\
(0,6,12)&(0,7,0)&(0,7,4)&(0,7,8)&(0,7,12)&(0,8,0)&(0,8,4)&(0,8,8)&(0,8,12)\\
(0,9,0)&(0,9,4)&(0,9,8)&(0,9,12)&(0,10,0)&(0,10,4)&(0,10,8)&(0,10,12)&(0,11,0)\\
(0,11,4)&(0,11,8)&(0,11,12)&(0,12,0)&(0,12,4)&(0,12,8)&(0,12,12)&(0,13,0)&(0,13,4)\\
(0,13,8)&(0,13,12)&(0,14,0)&(0,14,4)&(0,14,8)&(0,14,12)&(0,15,0)&(0,15,4)&(0,15,8)\\
(0,15,12)&(4,0,0)&(4,0,4)&(4,0,8)&(4,0,12)&(4,1,0)&(4,1,8)&(4,2,0)&(4,2,8)\\
(4,2,12)&(4,3,0)&(4,3,8)&(4,4,0)&(4,4,4)&(4,4,8)&(4,4,12)&(4,5,0)&(4,5,8)\\
(4,6,0)&(4,6,8)&(4,6,12)&(4,7,0)&(4,7,8)&(4,8,0)&(4,8,4)&(4,8,8)&(4,8,12)\\
(4,9,0)&(4,9,8)&(4,10,0)&(4,10,8)&(4,10,12)&(4,11,0)&(4,11,8)&(4,12,0)&(4,12,4)\\
(4,12,8)&(4,12,12)&(4,13,0)&(4,13,8)&(4,14,0)&(4,14,8)&(4,14,12)&(4,15,0)&(4,15,8)\\
(8,0,0)&(8,0,4)&(8,0,8)&(8,0,12)&(8,1,0)&(8,1,4)&(8,1,12)&(8,2,0)&(8,2,4)\\
(8,2,8)&(8,2,12)&(8,3,0)&(8,3,4)&(8,3,12)&(8,4,0)&(8,4,4)&(8,4,8)&(8,4,12)\\
(8,5,0)&(8,5,4)&(8,5,12)&(8,6,0)&(8,6,4)&(8,6,8)&(8,6,12)&(8,7,0)&(8,7,4)\\
(8,7,12)&(8,8,0)&(8,8,4)&(8,8,8)&(8,8,12)&(8,9,0)&(8,9,4)&(8,9,12)&(8,10,0)\\
(8,10,4)&(8,10,8)&(8,10,12)&(8,11,0)&(8,11,4)&(8,11,12)&(8,12,0)&(8,12,4)&(8,12,8)\\
(8,12,12)&(8,13,0)&(8,13,4)&(8,13,12)&(8,14,0)&(8,14,4)&(8,14,8)&(8,14,12)&(8,15,0)\\
(8,15,4)&(8,15,12)&(12,0,0)&(12,0,4)&(12,0,8)&(12,0,12)&(12,1,0)&(12,1,8)&(12,2,0)\\
(12,2,4)&(12,2,8)&(12,3,0)&(12,3,8)&(12,4,0)&(12,4,4)&(12,4,8)&(12,4,12)&(12,5,0)\\
(12,5,8)&(12,6,0)&(12,6,4)&(12,6,8)&(12,7,0)&(12,7,8)&(12,8,0)&(12,8,4)&(12,8,8)\\
(12,8,12)&(12,9,0)&(12,9,8)&(12,10,0)&(12,10,4)&(12,10,8)&(12,11,0)&(12,11,8)&(12,12,0)\\
(12,12,4)&(12,12,8)&(12,12,12)&(12,13,0)&(12,13,8)&(12,14,0)&(12,14,4)&(12,14,8)&(12,15,0)\\
(12,15,8)&&&&&&&&
\end{tabular}
\end{tiny}}

\noindent\underline{Matrices not representable as a diff. of two squares in $UT_2(\Z_{16})$, and $4\mid a,c$:}

{
\begin{tiny}
\begin{tabular}{ccccccccc}
\begin{tabular}{cccccccc}
(4,1,4)&(4,1,12)&(4,2,4)&(4,3,4)&(4,3,12)&(4,5,4)&(4,5,12)&(4,6,4)\\
(4,7,4)&(4,7,12)&(4,9,4)&(4,9,12)&(4,10,4)&(4,11,4)&(4,11,12)&(4,13,4)\\
(4,13,12)&(4,14,4)&(4,15,4)&(4,15,12)&(8,1,8)&(8,3,8)&(8,5,8)&(8,7,8)\\
(8,9,8)&(8,11,8)&(8,13,8)&(8,15,8)&(12,1,4)&(12,1,12)&(12,2,12)&(12,3,4)\\
(12,3,12)&(12,5,4)&(12,5,12)&(12,6,12)&(12,7,4)&(12,7,12)&(12,9,4)&(12,9,12)\\
(12,10,12)&(12,11,4)&(12,11,12)&(12,13,4)&(12,13,12)&(12,14,12)&(12,15,4)&(12,15,12)\\
\end{tabular}
\end{tabular}
\end{tiny}}\\

The previous tables  support 
Proposition~\ref{p1}, Lemma~\ref{lemma:mod16-classification} and Corollary \ref{cor:nonrep-mod16}.  In the case modulo $16$, we restrict to matrices whose diagonal entries are divisible by $4$,   since only in this case do new obstructions arise that are not visible in the modulo $4$ classification.

\textbf{Acknowledgements.} The authors were supported by the Croatian Science Foundation under the
project no. IP-2022-10-5008 (TEBAG). The authors acknowledge support from the
project “Implementation of cutting-edge research and its application as part of the
Scientific Center of Excellence for Quantum and Complex Systems, and Representations of Lie Algebras”, Grant No. PK.1.1.10.0004, co-financed by the European
Union through the European Regional Development Fund – Competitiveness and
Cohesion Programme 2021-2027. This research was funded by the European union:
NextGenerationEU through the National Recovery and Resilience Plan 2021-2026. Institutional grant of University of Zagreb Faculty of Science (IK IA 1.1.3. Impact4Math).  The authors would like to thank Professors Borka Jadrijevi\' c and Tomislav Pejkovi\' c  for carefully reading the manuscript and for their useful comments and suggestions.

\end{document}